\documentclass[a4paper, reqno]{amsart}

\usepackage[utf8]{inputenc}
\usepackage[T1]{fontenc}
\usepackage{mathtools}
\usepackage{amssymb}
\usepackage{enumitem}
\usepackage{tikz}
\usepackage{mathtools}
\usepackage{a4wide}
\usepackage{stmaryrd}

\usetikzlibrary{decorations.pathreplacing, patterns}
\usepackage[outline]{contour}
\contourlength{1.2pt}
\usepackage{hyperref}

\newcommand{\TODO}[1]%
{\par\fbox{\begin{minipage}{0.9\linewidth}\textbf{TODO:} #1\end{minipage}}\par}

\newtheorem{theorem}{Theorem}

\newtheorem{proposition}{Proposition}[section]
\newtheorem{corollary}[proposition]{Corollary}

\newtheorem{lemma}[proposition]{Lemma}
\theoremstyle{definition}
\newtheorem{definition}[proposition]{Definition}

\newtheorem{remark}[proposition]{Remark}

\DeclareMathOperator{\Int}{Int}

\DeclarePairedDelimiter{\norm}{\lVert}{\rVert}
\DeclareMathOperator{\mspec}{max-spec}
\newcommand{\maxspec}[1]{\mspec(#1)}

\newcommand{\calB}{\mathcal{B}}
\newcommand{\calP}{\mathcal{P}}
\newcommand{\calPone}{\mathcal{P}_1}
\newcommand{\calPtwo}{\mathcal{P}_2}

\newcommand{\calR}{\mathcal{R}}
\newcommand{\calS}{\mathcal{S}}
\newcommand{\calT}{\mathcal{T}}

\newcommand{\fixdiv}{\mathsf{d}}
\newcommand{\val}{\mathsf{v}}

\newcommand{\N}{\mathbb{N}}

\newcommand{\Z}{\mathbb{Z}}

\author{Sophie Frisch}
\address{Institut f\"ur Analysis und Zahlentheorie\\
  Graz University of Technology\\
  Kopernikusgasse 24\\8010 Graz\\Austria}
\email{\href{mailto:frisch@math.tugraz.at}{frisch@math.tugraz.at}}
\thanks{S.~Frisch is supported by the Austrian Science Fund (FWF):
P~30934}

\author{Sarah Nakato}
\address{Institut f\"ur Analysis und Zahlentheorie\\
  Graz University of Technology\\
  Kopernikusgasse 24\\8010 Graz\\Austria}
\email{\href{mailto:snakato@tugraz.at}{snakato@tugraz.at}}
\thanks{S.~Nakato is supported by the Austrian Science Fund (FWF):
P~27816}

\author{Roswitha Rissner}
\address{Institut f\"ur Mathematik\\Alpen-Adria-Universit\"at Klagenfurt\\
  Universit\"atsstra{\ss}e 65-67\\9020 Klagenfurt am W\"orthersee\\Austria}
\email{\href{mailto:roswitha.rissner@aau.at}{roswitha.rissner@aau.at}}
\thanks{R.~Rissner is supported by the Austrian Science Fund (FWF):
P~26114 and P~28466}

\title[Sets of lengths]{Sets of lengths of factorizations of
  integer-valued polynomials on Dedekind domains with finite residue
  fields}



\begin{document}

\begin{abstract}
  Let $D$ be a Dedekind domain with infinitely many maximal ideals,
  all of finite index, and $K$ its quotient field.  Let
  $\Int(D) = \{f\in K[x] \mid f(D) \subseteq D\}$ be the ring of 
  integer-valued polynomials on $D$.

  Given any finite multiset $\{k_1, \ldots, k_n\}$ of integers greater
  than $1$, we construct a polynomial in $\Int(D)$ which has exactly
  $n$ essentially different factorizations into irreducibles in
  $\Int(D)$, the lengths of these factorizations being $k_1$, \ldots,
  $k_n$.  We also show that there is no transfer homomorphism from the
  multiplicative monoid of $\Int(D)$ to a block monoid.
\end{abstract}

\maketitle

\section{Introduction}

By factorization we mean an expression of an element of a ring as a
product of irreducible elements.  Until not so long ago, the fact that
such a factorization, if it exists, need not be unique, was seen as a
pathology.
When mathematicians were shocked to find that uniqueness of
factorization does not hold in rings of integers in number fields,
they did not immediately study the details of this non-uniqueness, 
but moved on to unique factorization of ideals into prime ideals.
Non-uniqueness of factorization was avoided, whenever possible.

Only in the last few decades, some mathematicians, notably
Geroldinger and Halter-Koch~\cite{GeroldingerHalter:2006:NUF}, 
came around to the view that the precise details of non-uniqueness of
factorization  actually are a fascinating topic: the underlying 
phenomena give a lot of information about the arithmetic of a ring.

One important object of study is the set of lengths of factorizations
of a fixed element, cf.~\cite{Geroldinger:2016:sets-of-lengths}. 
The length of a factorization is the number of irreducible factors, and 
the set of lengths of an element is the set of all natural numbers that 
occur as lengths of factorizations of the element. 
Geroldinger and Halter-Koch \cite{GeroldingerHalter:2006:NUF} found that 
the sets of lengths of algebraic integers exhibit a certain structure.

In stark contrast to this, we show in Section~\ref{sec:proof} that every
finite set of natural numbers not containing $1$ occurs as the set of 
lengths of a polynomial in the ring of integer-valued polynomials on $D$,
\begin{equation*}
  \Int(D)= \{f\in K[x]\mid f(D)\subseteq D\},
\end{equation*}
where $D$ is a Dedekind domain with infinitely many maximal ideals,
all of them of finite index, and $K$ denotes the quotient field of $D$.
The special case of $D=\Z$ has been shown by 
Frisch~\cite{Frisch:2013:prescribed-sets}.

The study of non-uniqueness of factorization has mostly concentrated 
on Krull monoids so far. Krull monoids are characterized by having a 
``divisor theory''.
The multiplicative monoid $D\setminus \{0\}$ of an integral domain $D$
is Krull exactly if $D$ is a Krull ring,
cf.~\cite{GeroldingerHalter:2006:NUF}. 

The rings $\Int(D)$ for which we study non-uniqueness of factorization
are not Krull, but Pr\"ufer, 
cf.~\cite{CahenChabertFrisch:2000:interpolation, Loper:1998:intdpruefer}.
All factorizations of a single polynomial in $\Int(D)$, however, take 
place in a Krull monoid, namely, in the divisor-closed submonoid of
$\Int(D)$ generated by $f$.

Following Reinhart \cite{Reinhart:2014:monadic}, we call this monoid,
consisting of all divisors in $\Int(D)$ of all powers of $f$,
the {\em monadic submonoid generated by $f$}. 
That all monadic submonoids of $\Int(D)$ are Krull was
shown by Reinhart \cite{Reinhart:2014:monadic} in the case where $D$
is a unique factorization domain, and, by a different method, by Frisch
\cite{Frisch:2016:monadic} in the case where $D$ is a Krull ring.
Thus, our Theorem~\ref{thm:prescribed-lengths}, concerning non-unique
factorization in the Pr\"ufer ring $\Int(D)$, also serves to show
that quite wild factorization behavior is possible in Krull monoids.

Among Krull monoids, the best studied ones are multiplicative monoids 
of rings of algebraic integers. We should keep in mind, however, that 
the multiplicative monoids of rings of algebraic integers are very 
special, in that unique factorization of ideals always lurks in the 
background. In technical terms this means that there is a transfer 
homomorphism to a block monoid.

In Section~\ref{sec:transfer}, we show that there is no transfer
homomorphism to a block monoid from the multiplicative monoid of
$\Int(D)$.  This is relevant for two reasons: Firstly, because the
rings of whose multiplicative monoid it is known that it does not
admit such a transfer homomorphism are few and far between,
see~\cite{FanTringali:2017:power-monoids,
  GeroldingerSchmidZhong:2017:sets-lengths,
  GeroldingerSchwab:2017:ucfp}; and secondly, because most, if not
all, results so far concerning arbitrary finite sets occurring as sets
of lengths have been obtained using transfer homomorphisms to block
monoids~\cite{Kainrath:1999:FKM}.

Our main results are in Sections~\ref{sec:proof} and~\ref{sec:transfer}; 
in Section~\ref{sec:preliminaries} we introduce the necessary notation 
and Section~\ref{sec:aux-lemma} contains some useful lemmas.

\section{Preliminaries}
\label{sec:preliminaries}

We start with a short review of some elementary facts on 
factorizations, Dedekind domains and integer-valued polynomials,
and introduce some notation.

\subsection*{Factorizations}

We define here only the notions that we need throughout this paper, 
and refer to the monograph by 
Geroldinger and Halter-Koch~\cite{GeroldingerHalter:2006:NUF}
for a systematic introduction to non-unique factorizations.

Let $R$ be a commutative ring with identity and $r$, $s \in R$.
\begin{enumerate}
 \item If $r$ is a non-zero non-unit, we say $r$ is \emph{irreducible}
 in $R$ if it cannot be written as the product of two non-units of $R$.
 \item A \emph{factorization} of $r$ in $R$ is an expression
       \begin{equation}
       \label{eq:fac} r = a_{1}\cdots a_{n}
       \end{equation} 
       where $n\ge 1$ and $a_i$ is irreducible in $R$ for $1\le i \le n$.
 \item The number $n$ of irreducible factors is called the
 \emph{length} of the factorization in~\eqref{eq:fac}.
 \item The \emph{set of lengths} of $r$ is the set of all natural
 numbers $n$ such that $r$ has a factorization of length $n$.
 \item We say $r$ and $s$ are \emph{associated} in $R$ if there exists
 a unit $u \in R$ such that $r = us$. We denote this by $r \sim s$.
 \item Two factorizations of the same element,
       \begin{equation}\label{eq:2-fac-same-diff}
       r = a_{1}\cdots a_{n} = b_{1} \cdots b_{m},
       \end{equation} 
 are called \emph{essentially the same} if $n = m$ and, after reindexing, 
 $a_{j}\sim b_{j}$ for $1 \leq j \leq m$. 
 If this is not the case, the factorizations in
 \eqref{eq:2-fac-same-diff} are called \emph{essentially different}.
\end{enumerate}

\subsection*{Dedekind domains}

Recall that an integral domain $D$ is a \emph{Dedekind domain} if and
only if every non-zero ideal is a product of prime ideals. This is
equivalent to every non-zero ideal being invertible. It is also
equivalent to $D$ being a Noetherian domain such that the localization
at every non-zero maximal ideal is a discrete valuation domain. And it
is further equivalent to the following list of properties
\begin{enumerate}
\item $D$ is Noetherian
\item $D$ is integrally closed
\item $\dim(D) \le 1$
\end{enumerate}

From now on, we only consider Dedekind domains that are not fields.
For a Dedekind domain $D$ with quotient field $K$, let $\maxspec{D}$
denote the set of maximal ideals of $D$. Every prime ideal
$P\in\maxspec{D}$ defines a discrete valuation $\val_P$ by
$\val_P(a) = \max\{n\in \Z \mid a \in P^n\}$ for
$a\in K\setminus \{0\}$.
$\val_P$ is called the \emph{$P$-adic valuation} on $K$.

For a non-zero ideal $I$ of $D$, let
$\val_P(I) = \min\{\val_P(a) \mid a\in I\}$. This is compatible with
the definition of $\val_P(a)$ for $a\in K\setminus\{0\}$, in the sense 
that $\val_P(aD) = \val_P(a)$.
With this notation, the factorization of $I$ into prime ideals is
\begin{equation}\label{eq:dedekind-ideal-fac}
  I = \prod_{P\in\maxspec{D}} P^{\,\val_P(I)}
\end{equation}

Note that $\val_P(I) > 0$ is equivalent to $I\subseteq P$. There are 
only finitely many prime overideals of $I$ in $D$ and hence the product 
in Equation~\eqref{eq:dedekind-ideal-fac} is finite.

For two ideals $I$ and $J$ of $D$, $I \subseteq J$ is equivalent to
$\val_P(J) \le \val_P(I)$ for all $P\in \maxspec{D}$. 
Note that $I \subseteq J$ is equivalent to the fact that there exists 
an ideal $L$ of $D$ such that $JL=I$, in which case we say that $J$ 
divides $I$ and write $J \mid I$. This last equivalence is often 
summarized as ``to contain is to divide.''

For a thorough introduction to Dedekind domains, we refer to
Bourbaki~\cite[Ch.~VII,~§~2]{Bourbaki:1989:comm-alg}.

\subsection*{Dedekind domains with finite residue fields}

Let $D$ be a Dedekind domain. For a maximal ideal $P$ with finite
residue field we write $\norm{P}$ for $|D/P|$ and call this number the
\emph{index of $P$}. In what follows we will only consider Dedekind
rings with infinitely many maximal ideals, all of whose residue fields
are finite. We will frequently use the fact that there are only
finitely many maximal ideals of each individual finite index. This holds
in every Noetherian domain, as Samuel~\cite{Samuel:1971:eucl-rings} 
has shown; see also~Gilmer~\cite{Gilmer:1995:zero-dim-prods}.

We include a short proof by F.~Halter-Koch for the special case of
Dedekind domains.

\begin{proposition}[{Samuel~\cite{Samuel:1971:eucl-rings}, 
Gilmer~\cite{Gilmer:1995:zero-dim-prods}}]
  Let $D$ be a Dedekind domain. Then for each given $q\in \N$, there are
  at most finitely many maximal ideals $P$ of $D$ with $\norm{P}=q$.
\end{proposition}

\begin{proof}[Proof (Halter-Koch, personal communication).]

  Suppose that for some $q\ge 2$ there exist infinitely many prime
  ideals of index $q$, and let $0\neq a\in D$. Then there exist
  infinitely many prime ideals $P$ of $D$ such that $\norm{P} = q$ and
  $a \notin P$. For each such prime ideal $P$ we obtain
  $a^{q-1} \equiv 1 \mod P$, hence $a^{q-1} - 1 \in P$ and thus
  $a^{q-1}= 1$. So, every non-zero element of $D$ is a $(q - 1)$-st
  root of unity. Impossible!
\end{proof}

\subsection*{Integer-valued polynomials}
If $D$ is a domain with quotient field $K$, the \emph{ring of
  integer-valued polynomials} on $D$ is defined as
\begin{equation*}
  \Int(D) = \{f \in K[x] \mid f(D) \subseteq D\}.
\end{equation*}

Every non-zero $f \in K[x]$ can be written as a quotient
$f = \frac{g}{b}$ where $g\in D[x]$ and $b\in D\setminus\{0\}$.
Clearly, $f = \frac{g}{b}$ is an element of $\Int(D)$ if and only if
$b \mid g(a)$ for all $a \in D$.

\begin{definition}\label{def:fixdiv}
  Let $D$ be a domain and $g\in \Int(D)$. The \emph{fixed divisor} of
  $g$ is the ideal $\fixdiv(g)$ of $D$ generated by the elements $g(a)$ 
  with $a\in D$:
  \begin{equation*}
    \fixdiv(g) = (g(a) \mid a\in D)
  \end{equation*}
  We say that
  $g$ is \emph{image primitive} if $\fixdiv(g) = D$. By abuse of notation,
  this is also denoted $\fixdiv(g) = 1$.
\end{definition}

\begin{remark}\label{rem:intd-fixdiv}
  Let $D$ be a domain and $K$ its quotient field.
 \begin{enumerate}
 \item If $g\in D[x]$ and $b\in D\setminus\{0\}$, then $\frac{g}{b}$
   is an element of $\Int(D)$ if and only if
   $\fixdiv(g) \subseteq bD$.
   \label{intd-fixdiv-1}
 \item If $g\in D[x]$ and $P$ a prime ideal of $D$ such that
   $\fixdiv(g) \subseteq P$ then $g\in P[x]$ or
   $[D\colon P]\le\deg(g)$.%
   \label{rem:intd-fixdiv-index}
 \item If $f$, $g \in \Int(D)$, then
   $\fixdiv(fg)\subseteq \fixdiv(f)\fixdiv(g)$.
   \label{rem:fixdiv-prod}
 \item If $g\in D[x]$ is irreducible in $K[x]$, then every
   factorization of $g$ in $\Int(D)$ as a product of two (not
   necessarily irreducible) elements is of the form $c\frac{g}{c}$
   with $c\in D$ and $\fixdiv(g)\subseteq cD$.
   \label{rem:constant-factor}
 \item If $g\in D[x]$ is irreducible in $K[x]$ and $\fixdiv(g) = D$,
   then $g$ is irreducible in $\Int(D)$.
   \label{rem:irred-in-IntD}
 \end{enumerate}
\end{remark}

For a general introduction to integer-valued polynomials we refer to
the monograph by Cahen and Chabert~\cite{CahenChabert:1997:IVP} and to
their more recent survey paper~\cite{CahenChabert:2016:ivp-survey}.

\section{Auxiliary results}
\label{sec:aux-lemma}

In this section we develop tools to construct, first, split polynomials
in $D[x]$ with prescribed fixed divisor (Lemma~\ref{lemma:fix-div-val}), 
then, irreducible polynomials in $D[x]$ with prescribed fixed divisor 
(Lemma~\ref{lemma:replacements}), and, finally, polynomials of a
special form whose essentially different factorizations in $\Int(D)$ we
have complete control over (Lemma~\ref{lemma:factorizations}).

\begin{remark}\label{rem:multiset-op}
  In the following, we want to consider the multiplicity of roots of
  polynomials. For this purpose, we introduce some notation for
  multisets. Let $m_S(a)$ denote the multiplicity of an element $a$ in
  a multiset $S$ (with $m_S(a) = 0$ if $a\notin S$). For multisets $S$
  and $T$, let $S\uplus T$ denote the collection of elements $a$ in
  the union of the sets underlying $S$ and $T$ with multiplicities
  $m_{S\uplus T}(a) = m_S(a) + m_T(a)$ (the disjoint union of $S$ and $T$). 
  Note that $|S\uplus T| = |S| + |T|$.
\end{remark}

\begin{lemma}\label{lemma:fix-div-val}
  Let $D$ be a domain, $\calT\subseteq D$ a finite multiset and
  $f = \prod_{r\in \calT}(x-r)$.  If $Q$ is a non-zero prime ideal of
  $D$, then $\fixdiv(f) \subseteq Q$ if and only if $\calT$ contains a
  complete system of residues modulo $Q$.

  Furthermore, if $D$ is a Dedekind domain and
  $\calT = \calT_0  \uplus \biguplus_{i=1}^e\calT_i$ such that:
  \begin{enumerate}
  \item For all $1\leq i\leq e$, $\calT_i$ is a complete system of
    residues modulo $Q$ and the respective representatives of the same
    residue class in each $\calT_i$ are congruent modulo $Q^{2}$,
  \label{item:fdv-1}
  \item There exists $z \in D$ such that for all $s \in \calT_0$,
    $s \not \equiv z \mod Q$,
  \label{item:fdv-2}  
  \end{enumerate}
  then $\val_Q(\fixdiv(f)) = e$.
  
\end{lemma}
\begin{proof}  If $\calT$ does not contain a
  complete system of residues modulo $Q$, then there exists an element
  $a \in D$ such that $a \not \equiv r \mod Q$ for all $r \in \calT$.
  This implies $f(a) = \prod_{r\in \calT}(a-r) \not \in Q$, hence
  $\fixdiv(f) \nsubseteq Q$.

  Conversely, if $\calT$ contains a complete system of
  residues modulo $Q$ then, for all $a \in D$, there exists
  $r \in \calT$ such that $a \equiv r \mod Q$. This implies
  $f(a) = \prod_{r\in \calT}(a-r) \in Q$ for all $a \in D$ and hence
  $\fixdiv(f)\subseteq Q$.

  Now assume that $D$ is a Dedekind domain and
  $\calT = \biguplus_{i=1}^e\calT_i \uplus S$ such that
  \ref{item:fdv-1} and \ref{item:fdv-2} hold.  If
  $f_i = \prod_{r \in \calT_i}(x-r)$ for $1 \leq i \leq e$ and
  $g = \prod_{s \in \calT_0}(x-s)$, then $f = (\prod_{i=1}^ef_i)g$. Since
  $\calT_i$ is a complete system of residues modulo $Q$, it follows
  that $\val_Q(f_i(a)) \geq 1$ for all $a \in D$.  Therefore, for all
  $a \in D$,
  \begin{equation}\label{eq1:lemma:fix-div-val}
    \val_Q(f(a)) =\sum_{i=1}^e\val_Q(f_i(a)) + \val_Q(g(a)) \geq e
  \end{equation}

  For $1 \leq i \leq e$, let $a_i \in \calT_i$ with
  $a_i \equiv z \mod Q$. Since the elements $a_i$ are in the same
  residue class modulo $Q^2$, there exists $d\in D$ in the same
  residue class modulo $Q$ as $z$ and all the $a_i$, but in a different
  residue class modulo $Q^2$ from all the $a_i$.

  For such a $d$, then $\val_Q(f_i(d)) = 1$ for all $1\le i\le e$ and
  $\val_Q(g(d)) = 0$, since for all $s \in \calT_0$,
  $s \not \equiv z \equiv d \mod Q$. Therefore
  \begin{equation*}
  \val_Q(f(d)) =\sum_{i=1}^e\val_Q(f_i(d)) + \val_Q(g(d)) = e
  \end{equation*}
  which implies that $\val_Q(\fixdiv(f)) = e$.
\end{proof}

Next, we need to discuss how to replace split monic polynomials in $D[x]$ 
by monic polynomials in $D[x]$ which are irreducible in $K[x]$, without
changing the fixed divisors.

\begin{lemma}\label{lemma:replacements}
  Let $D$ be a Dedekind domain with infinitely many maximal ideals and
  $K$ its quotient field. Let $I \neq \emptyset$ be a finite set
  and $f_i\in D[x]$ be monic polynomials for $i\in I$.

  Then, there exist monic polynomials $F_i \in D[x]$ for $i\in I$,
  such that
\begin{enumerate}
\item $\deg(F_i) = \deg(f_i)$ for all $i\in I$,
\item the polynomials $F_i$ are irreducible in $K[x]$ and 
   pairwise non-associated in $K[x]$ and
\item for all subsets $J\subseteq I$ and all partitions
  $J = J_1 \uplus J_2$,
  \begin{equation*}
    \fixdiv\left(\prod_{j\in J_1} f_j\prod_{j\in J_2}F_j\right)
    = \fixdiv\left(\prod_{j\in J} f_j\right).
 \end{equation*}
 \label{repl-prop-3}
\end{enumerate}
\end{lemma}

\begin{proof}
  Let $P_1, \ldots, P_n$ be all maximal ideals $P$  of $D$ with
   $\norm{P}\le\deg\left(\prod_{i\in I} f_i\right)$.
  Suppose the prime factorization of the fixed divisor of the product
  of the $f_i$ is
  \begin{equation*}
    \fixdiv\left(\prod_{i\in I}f_i\right) =\prod_{j=1}^n P_j^{e_j}.
  \end{equation*}

  Let $Q \in \maxspec{D}\setminus \{P_1, \ldots, P_n\}$.  Using the
  Chinese Remainder Theorem, we add elements to the coefficients of
  the $f_i$ such that the resulting polynomials can be seen to be
  irreducible according to Eisenstein's irreducibility criterion with
  respect to $Q$, while retaining all relevant properties with respect
  to sufficiently high powers of the $P_i$.

  Let $f_{ik}$ denote the coefficient of $x^k$ in $f_i$.  For $i\in I$
  and $0\le k < \deg(f_i)$, let $g_{ik}\in D$ such that
  \begin{enumerate}
  \item $g_{ik} \in \prod_{j=1}^nP_j^{e_j+1}$ for all
    $0\le k < \deg(f_i)$.
  \item $g_{ik} \equiv -f_{ik} \mod Q$ for all
    $0\le k < \deg(f_i)$ and 
  \item
   $g_{i0} \not\equiv -f_{i0} \mod Q^2$.
  \end{enumerate}
 
  Since the $g_{ik}$ satisfying the above conditions are only
  determined modulo $Q^2\prod_{i=1}^n P_i^{e_i+1}$, there are infinitely
  many choices for each $g_{ik}$. We use this flexibility to 
  implement that $g_{i0}+f_{i0}\neq g_{j0}+f_{j0}$ for $i\neq j$. 
  Then, for $i\in I$, we set
  \[F_i = f_i +\sum_{k=0}^{\deg(f_i)-1}g_{ik}x^k.\]
  As the resulting $F_i$ are monic and distinct, they are pairwise 
  non-associated in $K[x]$.
%

  According to Eisenstein's irreducibility criterion, the polynomials
  $F_i$ are irreducible in $D[x]$ for $i\in I$,
  cf.~\cite[§29,~Lemma~1]{Matsumura:1989:comm-alg}. Since the $F_i$ are
  monic and $D$ is integrally closed, it follows that the $F_i$ are
  irreducible in $K[x]$ for all $i\in I$, cf.~\cite[Ch.~5,~§1.3,
  Prop.~11]{Bourbaki:1989:comm-alg}.

  By construction,
  \[F_i \equiv f_i \mod \left(\prod_{j=1}^nP_j^{e_j+1}\right)D[x]\]
  for all $i\in I$. 
  Now, if $g(x)$ is the product of any selection of the polynomials 
  $f_i$, and $G(x)$ the modified product in which some of the $f_i$ 
  have been replaced by $F_i$, then $g(x)$ is congruent to $G(x)$ 
  modulo $\left(\prod_{j=1}^nP_j^{e_j+1}\right)D[x]$.

  Hence, for all $a\in D$, $g(a)\equiv G(a)$ modulo 
  $\left(\prod_{j=1}^nP_j^{e_j+1}\right)$ and, therefore,
  \[\min_{a\in D}\val_P(G(a)) = \min_{a\in D}\val_P(g(a))\]
  for all $P$ that
  could conceivably divide the fixed divisor of $G(x)$ or $g(x)$
   by Remark~\ref{rem:intd-fixdiv}.\ref{rem:intd-fixdiv-index}.
  This implies the last assertion of the Lemma, to the effect that 
  substituting $F_i$ for some or all of the $f_i$ does not change 
  the fixed divisor of a product.
\end{proof}

Finally, the last two lemmas enable us to understand all essentially
different factorizations of a certain type of polynomials in $\Int(D)$.

\begin{lemma}\label{lemma:pre-factorizations}
  Let $D$ be a Dedekind domain with quotient field $K$ and
  $f\in \Int(D)$ of the following form:
  \[f = \frac{\prod_{i\in I}f_i}{c}\quad{\text with}\quad 
  \fixdiv\left(\prod_{i \in I}f_i\right) = cD,\] 
  where $c$ is a non-unit of $D$ and for each $i\in I$, $f_i\in D[x]$ 
  is irreducible in $K[x]$.

  Let $\calP \subseteq \maxspec{D}$ be the finite set of prime ideal
  divisors of $cD$. 
  If $f=g_1 \cdots g_m$ is a factorization of $f$ into (not necessarily
  irreducible) non-units in $\Int(D)$ then each $g_j$ is of the form
  \begin{equation*}
    g_j= a_j\prod_{i \in I_j}f_i,
  \end{equation*}
  where $\emptyset\neq I_j\subseteq I$ and $a_j\in K$, such that
  $I_1 \uplus \ldots \uplus I_m = I$, $a_1 \cdots a_m = c^{-1}$ and
\begin{enumerate}
\item
$\val_P(a_j) \le 0$
for all $P\in \maxspec{D}$ and all $1\le j\le m$; and
\item
$\val_P(a_j) = 0$
for all $P\in \maxspec{D}\setminus\calP$ and all $1\le j\le m$.
\end{enumerate}

\end{lemma}

\begin{proof}
  Let $f=g_1 \cdots g_m$ be a factorization of $f$ into (not necessarily
  irreducible) non-units in $\Int(D)$.  Since $\fixdiv(f)=1$, no $g_i$
  is a constant, by Remark \ref{rem:intd-fixdiv}.\ref{rem:constant-factor}.
  Each factor $g_j$ is, therefore, of the form
  \begin{equation}\label{eq:gen-fac}
    g_j= a_j\prod_{i \in I_j}f_i
  \end{equation}
  where $I_j$ is a non-empty subset of $I$ and $a_j\in K$, such that
  $I_1 \uplus \ldots \uplus I_m = I$ and $a_1 \cdots a_m = c^{-1}$. 
  Note that for all $P\in \maxspec{D}$
  \begin{equation}\label{eq:val-sum}
    \sum_{j=1}^m\val_P(a_j) =  -\val_P(c).
  \end{equation}
  Suppose  $\val_P(a_t) > 0$ for some maximal ideal $P$ and some
  $1\le t\le m$. Then
  $\sum_{j\neq t}\val_P(a_j) < - \val_P(c)$.

  Remark~\ref{rem:intd-fixdiv}.\ref{rem:fixdiv-prod} and the fact that 
  $\val_P\left(\fixdiv\left(\prod_{i\in I}f_i\right)\right) = \val_P(c)$
  imply
  $\val_P\left(\fixdiv\left(\prod_{j\neq t}\prod_{i\in I_j}f_i\right)\right)
  \le \val_P(c)$. 
  But now
  \begin{equation*}
    \val_P\left(\fixdiv\left(\prod_{j\neq t}g_j\right)\right) =
    \val_P\left(\fixdiv\left(\prod_{j\neq t} \prod_{i\in I_j}f_i\right)\right)
    + \sum_{j\neq t}\val_P(a_j) <0,
  \end{equation*}
  which means that
  \begin{equation*}
    \prod_{j\neq t} g_j   \notin \Int(D),
  \end{equation*}
  a contradiction.  We have established that  $\val_P(a_j) \le 0$ for
  all $P\in \maxspec{D}$ and all $1\le j\le m$. 
  Now Equation~\eqref{eq:val-sum}
  and the fact that $\val_P(c)=0$ for all $P\notin\calP$ imply 
  $\val_P(a_j) = 0$ for all $P\notin\calP$ and all $1\le j\le m$.
\end{proof}

\begin{definition}\label{def:indispensable}
  Let $D$ be a Dedekind domain, $f_i \in D[x]$ with $i\in I$ for a
  finite set $I\neq \emptyset$ and $\calP \subseteq \maxspec{D}$ be
  the finite set of prime ideal divisors of
  $\fixdiv\left(\prod_{i\in I}f_i\right)$.  If
  $P\in \calP$, we say $f_{k}$ is \emph{indispensable for $P$} (among
  the polynomials $f_i$ with $i\in I$) if for all $J\subseteq I$
  \begin{equation*}
  \val_P\left(\fixdiv\left(\prod_{i\in J}f_i\right)\right) > 0
 \Longrightarrow k\in J.
\end{equation*}
\end{definition}

\begin{remark}\label{remark:indispensable}
  Note that (with the notation of Definition~\ref{def:indispensable})
  $f_k$ is indispensable for $P\in \calP$ (among the polynomials $f_i$
  with $i\in I$) if and only if there exists an element $z\in D$ such
  that $\val_P(f_k(z)) > 0$ and $\val_P(f_i(z)) = 0$ for all
  $i\neq k$.
\end{remark}

\begin{remark}\label{remark:indispensable-replacements}
  For a finite set $I\neq \emptyset$, let $f_i$ and $F_i\in D[x]$ with $i\in I$ such
  that for all $J\subseteq I$ and all partitions $J = J_1 \uplus J_2$
  \begin{equation*}
    \fixdiv \left( \prod_{j\in J_1}f_i \prod_{j\in J_2} F_j\right) =    \fixdiv \left( \prod_{j\in J}f_i \right).
  \end{equation*}
  It follows that
  $\fixdiv\left(\prod_{i\in I}f_i\right) = \fixdiv\left(\prod_{i\in
      I}F_i\right)$ which implies that the fixed divisor of
  $\prod_{i\in I}f_i$ and $\prod_{i\in I}F_i$ have the same set $\calP$
  of prime ideal divisors. Moreover, for all $P\in \calP$ and all
  $J\subseteq I$, it follows that
  $\val_P\left(\fixdiv\left(\prod_{i\in J}F_i\right)\right) =
  \val_P\left(\fixdiv\left(\prod_{i\in J}f_i\right)\right)$. Hence,
  for $P\in \calP$, $f_k$ is indispensable for $P$ (among the
  polynomials $f_i$ with $i\in I$) if and only if $F_k$ is
  indispensable for $P$ (among the polynomials $F_i$ with $i\in I$).
  Note that this applies in particular in the setting of
  Lemma~\ref{lemma:replacements}.
\end{remark}

\begin{lemma}\label{lemma:factorizations}
  Let $D$ be a Dedekind domain with quotient field $K$ and
  $f\in \Int(D)$ of the following form:
  \[f = \frac{\prod_{i\in I}f_i}{c}\quad{\text with}\quad 
  \fixdiv\left(\prod_{i \in I}f_i\right) = cD,\] 
  where $c$ is a non-unit of $D$ and for each $i\in I$, $f_i\in D[x]$ 
  is irreducible in $K[x]$.
  Let $\calP \subseteq \maxspec{D}$ be the finite set of prime ideal
  divisors of $cD$. 

Suppose, for each $P\in \calP$, $\Lambda_P$ is a subset of $I$ such 
that $f_i$ is indispensable for $P$ for each $i\in\Lambda_P$. 
Let $\Lambda = \bigcup_{P\in \calP} \Lambda_P$.  

If $\bigcap_{P \in \calP} \Lambda_P \neq \emptyset$
then all essentially different factorizations of $f$ into irreducibles 
in $\Int(D)$ are given by:
\begin{equation*}
\frac{\left(\prod_{i \in \Lambda\cup J_1}f_i\right)}{c} 
\cdot \prod_{j\in J_2}f_j
\end{equation*}
(each $f_j$ with $j\in J_2$ counted as an individual factor),
where $I = \Lambda \uplus J_1 \uplus J_2$ such that $J_1$ is minimal
with
$\fixdiv\left(\prod_{i \in \Lambda\cup J_1}f_i\right) = cD$.
\end{lemma}

\begin{proof}
  Let $f=g_1 \cdots g_m$ be a factorization of $f$ into (not necessarily
  irreducible) non-units in $\Int(D)$. As in
  Lemma~\ref{lemma:pre-factorizations},
  \begin{equation}
    g_j= a_j\prod_{i \in I_j}f_i
  \end{equation}
  where $I_j$ is a non-empty subset of $I$ and $a_j\in K$, such that
  $I_1 \uplus \ldots \uplus I_m = I$ and $a_1 \cdots a_m = c^{-1}$.
  Furthermore, 
  $\val_P(a_j) \le 0$ for all $P\in \maxspec{D}$ and all $1\le j\le m$
  and $\val_P(a_j) = 0$ for all $P\notin \calP$ and all $1\le j\le m$.

  We know there exists a polynomial $f_{i_0}$ that is indispensable
  for all $P\in \calP$.  We may assume that $i_0\in I_1$. 
  By the definition of indispensable polynomial, 
  $\val_P\left(\fixdiv\left(\prod_{i\in I_j} f_i \right)\right)= 0$,
  for $2\le j \le m$ and all $P\in \calP$.
  From this and the fact that $g_j = a_j\prod_{i\in I_j}f_i$ 
  is in $\Int(D)$, we infer that  $\val_P(a_j) = 0$ for
  all $2\le j \le m$ and all $P\in \calP$. We have shown that
  $a_2$, \ldots, $a_m$ are units of $D$.
  
  Now $u=a_2\cdots a_m$ is a unit of $D$ such that $a_1 u=c^{-1}$. 
  Since $g_1\in\Int(D)$, we must have
  \begin{equation*}
   \val_P\left( \fixdiv\left( \prod_{i\in I_1}f_i \right) \right) = \val_P(c) >0
  \end{equation*}
  for all $P\in\calP$ and, by Definition~\ref{def:indispensable}, 
  $\Lambda \subseteq I_1$.
  
  So far we have shown that every factorization $f=g_1\cdots g_m$ of
  $f$ into (not necessarily irreducible) non-units of $\Int(D)$ is --
  up to reordering of factors and multiplication of factors by units
  in $D$ -- the same as one of the following:
  \begin{equation}\label{eq:fact-into-non-units}
  \frac{\left(\prod_{i\in \Lambda\cup J_1}f_i\right)}{c}
  \cdot \left(\prod_{j\in I_2}f_j\right)\cdots 
  \left(\prod_{j\in I_m}f_j\right),
  \end{equation}
  where $I = I_1\uplus \cdots \uplus I_m$ and $I_1 = \Lambda \uplus J_1 $.

  It remains to characterize, among the factorizations of the above
  form, those in which all factors are irreducible in $\Int(D)$.

  Since $\fixdiv(f) = D$, it is clear that $\fixdiv( g_j) = D$ for all 
  $1\leq j \leq m$, 
  by Remark~\ref{rem:intd-fixdiv}.\ref{rem:fixdiv-prod}. By the
  same token, $\fixdiv(f_i) =D$ for all $i\in I_j$ with $j\ge 2$.
  Since the $f_i$ are irreducible in $K[x]$, those of them with
  fixed divisor $D$ are irreducible in $\Int(D)$, by
  Remark~\ref{rem:intd-fixdiv}.\ref{rem:irred-in-IntD}.
  The criterion for each factor $g_j=\prod_{i\in I_j}f_i$ with $j\ge 2$ 
  to be irreducible is, therefore, $|I_j|=1$ for all $j\ge 2$. 
  
  Now, concerning the irreducibility of $g_1$, the same arguments 
  that lead to Equation~\eqref{eq:fact-into-non-units}, applied to 
  $g_1=c^{-1}\left(\prod_{i\in \Lambda\cup J_1}f_i\right)$
  instead of $f$, show that $g_1$ is irreducible in $\Int(D)$ if 
  and only if we cannot split off any factors $f_i$ with $i\in J_1$. 
  This is equivalent to $\fixdiv\left(\prod_{i\in \Lambda\cup J}f_i\right) \neq cD$ 
  for every proper subset $J\subsetneq J_1$, in other words, to
  $J_1 $ being minimal such that
  $\fixdiv\left(\prod_{i\in \Lambda\cup J_1}f_i\right) =cD$.
  In this case we set $J_2 = \bigcup_{j=2}^m I_j$ and the assertion follows.
\end{proof}

\begin{remark}
  When $\left|\calP\right|>1$, the hypothesis
  $\bigcap_{P\in \calP} \Lambda_P \neq \emptyset$ in
  Lemma~\ref{lemma:factorizations} can be replaced by a weaker
  condition:

  Consider the prime ideals $P\in \calP$ as vertices of an undirected
  graph $G$ and let $(P, Q)$ be an edge of $G$ if and only if there 
  exists a polynomial $f_t$ which is indispensable for both $P$ and $Q$. 
  If $G$ is a connected graph, then the conclusion of
  Lemma~\ref{lemma:factorizations} holds. The proof of 
  Lemma~\ref{lemma:factorizations} generalizes readily.
\end{remark}

\section{Construction of polynomials with prescribed sets of lengths}
\label{sec:proof}

We are now ready to prove the main result of this paper.

\begin{theorem}\label{thm:prescribed-lengths}
  Let $D$ be a Dedekind domain with infinitely many maximal ideals,
  all of them of finite index.

  Let $1\le m_1 \le m_2 \le \cdots \le m_n$ be natural numbers.

  Then there exists a polynomial $H\in \Int(D)$ with 
  exactly $n$ essentially different factorizations 
  into irreducible polynomials in $\Int(D)$,
  the length of these factorizations being $m_1 + 1$, \ldots, $m_n + 1$.
\end{theorem}

\begin{proof}
  If $n=1$, then $H(x) = x^{m_1+1}\in \Int(D)$ is a polynomial which
  has exactly one factorization, and this factorization has length
  $m_1+1$.  From now on, assume $n\ge 2$.

  First, we construct $H(x)$.  Let
  $N = \left(\sum_{i=1}^n m_i\right)^2 - \sum_{i=1}^n m_i^2$ and $P$ a
  prime ideal of $D$ with $\norm{P}>N+1$. Let $c\in D$ such that
  $\val_P(c)=1$ and $c$ is not contained in any maximal ideal of index $2$.

  Say the prime factorization of $cD$ is $cD=P Q_1^{e_1}\cdots Q_t^{e_t}$.
Let $\tau=(\norm{P}-N)$ and $\sigma$ the maximum of the following numbers:
$\tau$, and $e_i\norm{Q_i}$ for $1\le i\le t$.

  We now choose two subsets of $D$: a set $\calR$ of order $N$,
  and $\calS=\{s_0,\ldots, s_{\sigma-1}\}$.
  Using the Chinese Remainder Theorem, we arrange that $\calR$ and 
  $\calS$ have the following properties:
\begin{enumerate}
\item $s_0 \equiv 0 \mod P$, and
  $\{s_0,\ldots, s_{\tau-1}\}\cup \calR$ is a complete system of
  residues modulo $P$.
\item $s_i\equiv 0\mod P$ for all $i\ge \tau$.
\item For each $Q_i$, $\calS$ contains $e_i$ disjoint complete systems
  of residues, in which the respective representatives of the same
  residue class in different systems are congruent modulo $Q_i^{2}$.
\item For each $Q_i$, no more than $e_i$ elements of $\calS$ are
  congruent to $1$ modulo $Q_i$.
\item For all $r\in\calR$, $r\equiv 0 \mod \bigcap_{i=1}^tQ_i$.
\item $\calR \cup \calS$ does not contain a complete system of
  residues for any prime ideal $Q$ of $D$
   other than $P$ and $Q_1,\ldots, Q_t$.
\end{enumerate}

We now assign indices to the elements of $\calR$ as follows
\begin{equation*}
  \calR = \{r_{(k,i,h,j)} \mid 
  1\le k, h \le n, k\ne h, 1\le i\le m_k, 1\le j\le m_h\}.
\end{equation*}
This allows us to visualize the elements of $\calR$ as entries in a
square matrix $B$ with $m =\sum_{i=1}^n m_i$ rows and columns, in
which the positions in the blocks of a block-diagonal matrix with
block sizes $m_1,\ldots, m_n$ are left empty, see
Figure~\ref{fig:matrix}.

The rows and columns of $B$ are divided into $n$ blocks each, such 
that the $k$-th block of rows consists of $m_k$ rows, and similarly for
columns. Now $r_{(k,i,h,j)}$ designates the entry in row $(k,i)$, that
is, in the $i$-th row of the $k$-th block of rows, and in column
$(h,j)$, that is, in the $j$-th column of the $h$-th block of
columns. Since no element of $\calR$ has row and column index in the
same block, the positions of a block-diagonal matrix with blocks of
sizes $m_1,\ldots, m_n$ are left empty.

For $1\le k\le n$, let $I_k = \{(k,i) \mid 1\le i \le m_k\}$ and set
\begin{equation*}
  I = \bigcup_{k=1}^n I_k.
\end{equation*}
Then
\begin{equation*}
  I = \{(k,i)\mid 1\le k \le n, 1\le i \le m_k\}
\end{equation*} is the set of all possible row indices, or,
equivalently, column indices.

For $(k,i)\in I_k$, let $B[k,i]$ be
the set of all elements $r\in \calR$ which are either in row or in
column $(k,i)$ of $B$, that is,
\begin{equation}\label{eq:def-Bki}
  B[k,i] = \{r_{(k,i,h,j)} \mid (h,j)\in
  I\setminus I_k\}\cup \{r_{(h,j,k,i)} \mid (h,j)\in I\setminus I_k\}
\end{equation}

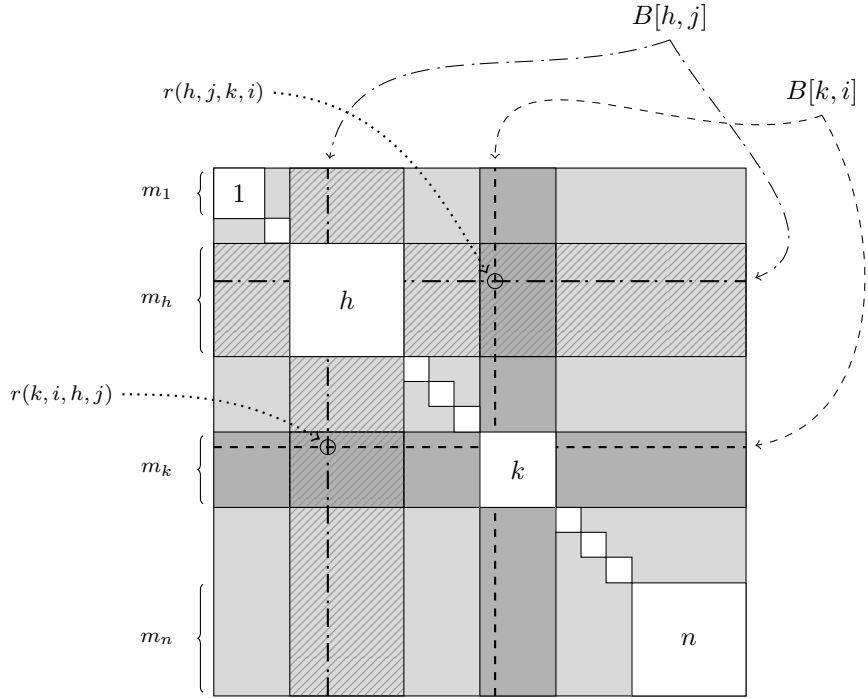
\begin{figure}[h]
 \begin{tikzpicture}

\newcommand{\Rect}[6]{ \draw[#1] (#2,#3) rectangle(#2+#4,#3-#5) {#6};
} \newcommand{\block}[5]{ \Rect{#1}{#2}{#3}{#4}{#4}{#5}; }

\block{fill=black!15}{2}{7}{7}{}

\block{fill=white}{2}{7}{0.67}{node[pos=.5] {$1$}}
\block{fill=white}{7.5}{1.5}{1.5}{node[pos=.5] {$n$}}

\Rect{fill=black!30}{2}{3.5}{7}{1}{};
\Rect{fill=black!30}{5.5}{7}{1}{7}{};

\Rect{fill=black!5, pattern=north east lines, pattern
  color=black!40}{2}{6}{7}{1.5}{};
\Rect{fill=black!5, pattern=north
east lines, pattern color=black!40}{3}{7}{1.5}{7}{};
%
\Rect{fill=black!20, pattern=north east lines,
  pattern color=black!40}{5.5}{6}{1}{1.5}{};
\Rect{fill=black!20,
pattern=north east lines, pattern color=black!40}{3}{3.5}{1.5}{1}{};

\draw[thick,dashed] (5.7,0) -- (5.7,7);
\draw[dashed, thick] (2,3.3) -- (9,3.3);
\draw[dash pattern={on 7pt off 2pt on 1pt off3pt}, thick] (2,5.5) -- (9,5.5);
\draw[dash pattern={on 7pt off 2pt on 1pt off 3pt}, thick] (3.5,0) -- (3.5,7);

\block{fill=white}{5.5}{3.5}{1}{node[pos=.5] {$k$}}
\block{fill=white}{3}{6}{1.5}{node[pos=.5] {$h$}}
\block{fill=white}{2.67}{6.335}{0.33}{}
\block{fill=white}{4.5}{4.5}{0.33}{}
\block{fill=white}{4.83}{4.17}{0.33}{}
\block{fill=white}{5.16}{3.84}{0.34}{}
\block{fill=white}{6.5}{2.5}{0.33}{}
\block{fill=white}{6.83}{2.17}{0.33}{}
\block{fill=white}{7.16}{1.84}{0.34}{}

\draw[decorate,decoration={brace,amplitude=2pt},xshift=-4pt,yshift=0pt]
(2,6.38) -- (2,6.95) node [black,midway,xshift=-0.6cm] {\footnotesize
$m_1$}; \draw [decorate,decoration={brace,amplitude=2pt,
},xshift=-4pt,yshift=0pt] (2,4.55) -- (2,5.95) node
[black,midway,xshift=-0.6cm] {\footnotesize $m_h$};
\draw[decorate,decoration={brace,amplitude=2pt},xshift=-4pt,yshift=0pt]
(2,2.55) -- (2,3.45) node [black,midway,xshift=-0.6cm] {\footnotesize
  $m_k$};
\draw[decorate,decoration={brace,amplitude=2pt},xshift=-4pt,yshift=0pt]
 (2,0.05) -- (2,1.48) node [black,midway,xshift=-0.6cm] {\footnotesize
 $m_n$};

\node (B2j) at (8,9) {$B[h,j]$};
\node (B2jn) at (3.5,7){};
\node (B2je) at (9,5.5) {};

\draw (B2j.south) edge[dash pattern={on 7pt off 2pt on 1pt off 3pt},
out=200,in=80,->] (B2jn);

\draw (B2j.south)edge[dash pattern={on 7pt off 2pt on 1pt off
  3pt},out=300,in=20,->] (B2je);

\node (Bki) at (10,8) {$B[k,i]$};

\node (Bkin) at (5.7,7) {};

\node (Bkie) at (9,3.3) {};

\draw (Bki.south) edge[dashed,out=200,in=90,->] (Bkin);

\draw (Bki.south) edge[dashed,out=300,in=10,->] (Bkie);



\node[draw, circle, inner sep=2] (i1) at (5.7,5.5) {};

\node[draw, circle, inner sep=2] (i2) at (3.5,3.3) {};

\node (r1) at (2,8) {\footnotesize $r(h,j,k,i)$};

\draw (r1.east) edge[dotted,thick,out=355,in=120,->] ([xshift=-3pt,
yshift=3pt]i1);

\node (r2) at (0,4) {\footnotesize $r(k,i,h,j)$}; \draw (r2.east)
edge[dotted,thick, out=0,in=150,->] ([xshift=-3pt, yshift=3pt]i2);

\end{tikzpicture}  
 \caption{
   Say the $k$-th region of $B$ consists of the positions with either 
   column index or row index in the $k$-th block. 
   Then the union of the entries in any $n-1$ different regions covers 
   $\calR$.
   A union of different $B[u,v]$, from which $B[k,i]$ and $B[h,j]$ for
   two different blocks $k\neq h$ are missing, however, does not cover
   $\calR$, because $r_{(k,i,h,j)}$ and $r_{(h,j,k,i)}$ are not included.
 }
 \label{fig:matrix}
\end{figure}

In order to construct $H\in \Int(D)$, we set
$s(x)=\prod_{i=0}^{\sigma-1}(x-s_i)$ and, for $(k,i) \in I$,

\begin{equation*}
  f_i^{(k)}(x) = \prod_{r \in B[k,i]} (x-r).
\end{equation*}

Then, let $S(x)\in D[x]$, and, for each $(k,i) \in I$,
$F_i^{(k)}(x) \in D[x]$ be monic polynomials such as we know to exist
by Lemma~\ref{lemma:replacements}: irreducible in $K[x]$, pairwise
non-associated in $K[x]$, with $\deg(S)=\deg(s)$ and
$\deg(F_i^{(k)})= \deg(f_i^{(k)})$, and such that, for every selection
of polynomials from among $s$ and $f_i^{(k)}$ for $(k,i) \in I$, the
product of the polynomials has the same fixed divisor as the modified
product in which $s$ has been replaced by $S$ and each $f_i^{(k)}$ by
$F_i^{(k)}$. Now, let
\begin{equation*}
G(x) = S(x)\! \prod_{(k,i) \in I}\!\! F_i^{(k)}(x)
\quad\textrm{and}\quad
H(x) = \frac{G(x)}{c}.
\end{equation*}

Second, we show that $\fixdiv(G(x))=cD$, which implies 
$H(x)\in \Int(D)$ and $\fixdiv(H(x))=1$.
Note that
\begin{equation}\label{eq:fixdiv=c}
 \fixdiv(G(x)) =
 \fixdiv\left(S(x) \prod_{(k,i) \in I} F_i^{(k)}(x)\right) =
 \fixdiv\left( s(x) \prod_{(k,i)\in I}f_i^{(k)}\right) =
 \fixdiv\left(\prod_{i=0}^{\sigma - 1}(x-s_i)\prod_{r\in \calR}(x-r)^2\right).
\end{equation}

By construction, 
the multiset $\calR \uplus \calR \uplus \calS$ contains a complete
system of residues modulo $P$, and the residue class modulo $P$
of $s_1\in\calS$ occurs only once among the elements of 
$\calR \uplus \calR \uplus \calS$.
Equation~\eqref{eq:fixdiv=c} and
Lemma~\ref{lemma:fix-div-val}, applied to $Q=P$ and
$\calT = \calR \uplus \calR \uplus \calS$, $e=1$, and $z=s_1$, together
imply that
\begin{equation*}
  \val_P\left(\fixdiv\left(S(x) \prod_{(k,i) \in I}
      F_i^{(k)}(x) \right)\right) = 1
\end{equation*} 
One can argue similarly for $Q_i$, $1\le i \le t$:
The
multiset $\calR \uplus \calR \uplus \calS$ contains $e_i$ disjoint
complete systems of residues modulo $Q_i$ in which the respective
representatives of the same residue class in different systems are
congruent modulo $Q_i^{2}$. 
No more than $e_i$
elements of $\calR \uplus \calR \uplus \calS$ are congruent $1$ modulo
$Q_i$, and these $e_i$ elements are all in the same residue class modulo
$Q_i^{2}$.
By Lemma~\ref{lemma:fix-div-val}, applied to $Q=Q_i$,
$\calT=\calR \uplus \calR \uplus \calS$, $e=e_i$ and $z=1$, and
Equation~\eqref{eq:fixdiv=c}, it follows that
\begin{equation}
  \val_{Q_i}\left(\fixdiv\left(S(x) \prod_{(k,i) \in I}
      F_i^{(k)}(x) \right)\right) = e_i
\end{equation}
for $1\le i \le t$.  Since $\calR \uplus \calR \uplus \calS$ does not
contain a complete system of residues modulo any prime ideal other
than $P$, $Q_1$, \ldots, $Q_t$, we conclude (by
Lemma~\ref{lemma:fix-div-val}) that
\begin{equation*}
\fixdiv(G(x)) =
  \fixdiv\left(S(x) \prod_{(k,i) \in I}
    F_i^{(k)}(x)\right) =P Q_1^{e_1} \cdots Q_t^{e_t} = cD.
\end{equation*}
This shows $H(x)\in \Int(D)$ and $\fixdiv(H(x))=1$.

Third, we prove that the essentially different factorizations of $H(x)$ 
into irreducibles in $\Int(D)$ are given by:
\begin{equation}\label{eq:all-fact-of-H}
  H(x) =  F_1^{(h)}(x)\cdots F_{m_h}^{(h)}(x)\cdot\frac{S(x)\prod_{(k,i)\in
      I\setminus I_h} F_i^{(k)}(x)}{c}
\end{equation}
where $1 \le h \le n$.

It follows from the properties of $\calR$ and $\calS$ that the
polynomial $s(x)$ is indispensable for the prime ideals $P$ and $Q_1$,
\ldots, $Q_t$ (among the polynomials $s(x)$ and $f_i^{(k)}$ for
$(k,i)\in I$). This further implies that the polynomial $S(x)$ is
indispensable for the prime ideals $P$ and $Q_1$, \ldots, $Q_t$ (among
the polynomials $S(x)$ and $F_i^{(k)}$ for $(k,i)\in I$),
cf.~Remark~\ref{remark:indispensable-replacements}.

Thus, by Lemma~\ref{lemma:factorizations}, the essentially different
factorizations of $H(x)$ into irreducibles in $\Int(D)$ are given by:
\begin{equation}\label{eq:all-factorizations}
 H(x) =
\frac{S(x)\prod_{(k,i)\in J} F_i^{(k)}(x)}{c} \prod_{(h,j)\in
I\setminus J}F_j^{(h)}(x)
\end{equation} where $J\subseteq I$ is minimal such that
$\fixdiv\left(S(x)\prod_{(k,i)\in J}F_i^{(k)}(x)\right) = cD$.

Since $\val_{Q_i}(\fixdiv(S(x))) = e_i$ by
Lemma~\ref{lemma:fix-div-val}, the possible choices for
$J \subseteq I$ only depend on the prime ideal $P$. For a subset
$J\subseteq I$, let $\calB_J = \biguplus_{(k,i)\in J} B[k,i]$.  Then
\begin{equation}\label{eq:search-J1}
\fixdiv\left( S(x)\prod_{(k,i)\in J} F_i^{(k)}(x) \right) =
\fixdiv\left( \prod_{r\in \calS} (x-r) \prod_{r\in \calB_J} (x-r) \right)
\end{equation}
and it follows from Lemma~\ref{lemma:fix-div-val} that the fixed
divisor in Equation~\eqref{eq:search-J1} is equal $cD$ if and only if
$\calS \uplus \calB_J$ contains a complete set of residues modulo $P$
which is in turn equivalent to $\calR \subseteq \calB_J$. This is the
case if and only if there exists $1\le h \le n$ with
$I\setminus I_h \subseteq J$.

Therefore, $J\subseteq I$ is minimal with
$\fixdiv\left(S(x)\prod_{(k,i)\in J}F_i^{(k)}(x)\right) = cD$ if and
only if $J = I\setminus I_h$ for some $1\le h\le n$. 
Hence, the essentially different factorizations of $H(x)$, given by
Equation~\eqref{eq:all-factorizations}, are precisely the $n$ 
essentially different factorizations stated in
Equation~\eqref{eq:all-fact-of-H}, which are of lengths 
$m_1+1$, \ldots, $m_n+1$.
\end{proof}

\begin{corollary}\label{cor:kain}
  Let $D$ be a Dedekind domain with infinitely many maximal ideals,
  all of them of finite index.

  Then every finite subset of $\N\setminus \{1\}$ is the set of
  lengths of a polynomial in $\Int(D)$.
\end{corollary}

\begin{remark}
  Kainrath~\cite[Theorem~1]{Kainrath:1999:FKM} proved a similar
  result as Corollary~\ref{cor:kain} for Krull monoids $H$ with
  infinite class group in which every divisor class contains a prime
  divisor. In his proof, he uses transfer mechanisms.

  Corollary~\ref{cor:not-transfer-Krull} in the following section will
  show that this technique is not applicable to the proof of either
  Theorem~\ref{thm:prescribed-lengths} or Corollary~\ref{cor:kain}.
\end{remark}

\section{Not a transfer Krull domain}
\label{sec:transfer}

In this section we show that if $D$ is a Dedekind domain
with infinitely many maximal ideals, all of finite index, 
then there does not exist a transfer homomorphism from
the multiplicative monoid $\Int(D)\setminus \{0\}$ to a block monoid.
In the terminology introduced by 
Geroldinger~\cite{Geroldinger:2016:sets-of-lengths}, this means, 
$\Int(D)$ is not a \emph{transfer Krull domain}. 

We refer to~\cite[Definitions~2.5.5
\&~3.2.1]{GeroldingerHalter:2006:NUF} for the definition of a block
monoid and a transfer homomorphism, respectively. So far, there is
only a small list of examples of naturally occurring rings $R$ for
which it has been shown that there is no transfer homomorphism from
$R\setminus\{0\}$ to a block monoid,
see~\cite{FanTringali:2017:power-monoids,
  GeroldingerSchmidZhong:2017:sets-lengths,
  GeroldingerSchwab:2017:ucfp}.

In a block monoid, the lengths of factorizations of elements of the form 
$c\cdot d$ with $c$, $d$ irreducible, $c$ fixed, are bounded by a constant
depending only on $c$, cf.~\cite[Lemma~6.4.4]{GeroldingerHalter:2006:NUF}.
More generally, every monoid
admitting a transfer homomorphism to a block monoid has this
property; see \cite[Proposition~3.2.3]{GeroldingerHalter:2006:NUF}.

We now demonstrate for the irreducible element $c=x$ in $\Int(D)$ 
that the lengths of factorizations of elements of the form $c\cdot d$ 
with $d$ irreducible in $\Int(D)$ are not bounded. 
We infer from this that there does not exist a transfer homomorphism 
from the multiplicative monoid $\Int(D)\setminus \{0\}$ to a block monoid.

\begin{theorem}
  Let $D$ be a Dedekind domain with infinitely many maximal ideals,
  all of them of finite index.

  Then for every $n \geq 1$ there exist irreducible elements
  $H, G_1, \ldots, G_{n + 1}$ in $\Int(D)$ such that
  \[xH(x) = G_1(x) \cdots G_{n+1}(x).\]
\end{theorem}

\begin{proof}
  Let $P_1,\ldots, P_n$ be distinct maximal ideals of $D$, none of
  them of index $2$.  By $\val_i$ we denote the discrete valuation
  associated to $P_i$.  Let $c\in D$ such that $\val_i(c)=1$ for
  $i=1,\ldots,n$, and $c$ is not contained in any maximal ideal of $D$
  of index $2$.  Say the prime factorization of $cD$ is
  $cD=P_1\cdot\ldots\cdot P_n\cdot Q_1^{e_1}\cdot\ldots\cdot
  Q_m^{e_m}$, and define
  \[ N = \max\left(\{\norm{P_i}\mid 1\le i\le n\} \cup 
  \{e_i\norm{Q_i}\mid 1\le i\le m\}\right). \]

  Let $\calP=\{P_i\mid 1\le i\le n\}$,
  $\calPone=\{Q_i\mid 1\le i\le m\}$, and
\[\calPtwo=\{Q\in \maxspec{D}\setminus (\calP\cup \calPone)\mid
  \norm{Q}\le N+n\}.\]

Let $\calR$ be a subset of $D$ of order $N$ with the following
properties (which can be realized by the Chinese Remainder Theorem):
\begin{enumerate}
\item $\calR$ contains an element
  $r_0\in\left(\bigcap_{i=1}^n P_i\right)\cap \left(\bigcap_{i=1}^m
    Q_i^2\right)$.
\item No element of $\calR$ other than $r_0$ is in any $P_i\in \calP$.
\item For each $P_i\in\calP$, $\calR$ contains a complete system of
  residues modulo $P_i$.
\item For each $Q_i \in\calPone$, $\calR$ contains $e_i$ disjoint
  complete systems of residues, in which the respective
  representatives of the same residue class in different systems are
  congruent modulo $Q_i^{2}$;
\item No more than $e_i$ elements of $\calR$ are in $Q_i$.
\item For all $Q\in\calPtwo$, all elements of $\calR$ are contained in $Q$.
\end{enumerate}

We set $\calB = \calR\setminus \{r_0\}$.

Also, let $a_1, \ldots,a_n\in D$ with the following properties (which,
again, can be realized by the Chinese Remainder Theorem):
\begin{enumerate}
\item For all $i=1,\ldots, n$, $a_i \equiv 0 \mod P_i$.
\item For all $i=1,\ldots, n$, $a_i \equiv 1 \mod P_j$ for all $j\neq i$.
\item For all $Q\in \calP_1 $, $a_n\equiv 0 \mod Q^2$ and
  $a_i\equiv 1 \mod Q$ for all $1\le i < n$,
\item For all $Q\in \calP_2$ and all $1\le i\le n$, $a_i \equiv 0 \mod Q$.
\end{enumerate}

Let $f(x) = \prod_{b\in \calB}(x-b)$ and let $F(x)\in D[x]$ be monic
and irreducible in $K[x]$ such that for every selection of polynomials
from the set $\{x,f\} \cup \{(x-a_i) \mid 1 \leq i \leq n\}$ the
product of the polynomials has the same fixed divisor as the modified
product in which $f$ has been replaced by $F$, as in Lemma
\ref{lemma:replacements}.

Lemmas~\ref{lemma:replacements} and~\ref{lemma:fix-div-val}, applied
to $\calT=\calB \cup \{a_1,\ldots, a_n\}$ and each of the prime ideals
in $\calP \cup \calPone$, imply that
\begin{equation*}
  \fixdiv\left( F(x)\prod_{i=1}^n(x-a_i)\right) =
  \fixdiv\left( f(x)\prod_{i=1}^n(x-a_i)\right) = cD.
\end{equation*}
Similarly, Lemmas~\ref{lemma:replacements} and~\ref{lemma:fix-div-val},
applied to $\calT = \calB \cup \{0\}$ and each of the prime ideals in
$\calP \cup \calPone$, imply that
\begin{equation*}
  \fixdiv\left( xF(x)\right)
  =\fixdiv\left(xf(x)\right)
   = cD.
\end{equation*}

We set
\begin{equation*}
  H(x) =  \frac{F(x)\prod^{n}_{j=1}(x-a_{j})}{c}\quad\textrm{ and }\quad G(x) =
  \frac{xF(x)}{c}.
\end{equation*}

Then $G(x)$ and $H(x)$ are elements of $\Int(D)$ with
$\fixdiv(G(x)) = \fixdiv(H(x)) = 1$ such that
\begin{equation*}
  xH(x) = G(x)(x-a_1) \cdots (x-a_n).
\end{equation*}

It remains to show that $H(x)$ and $G(x)$ are irreducible in
$\Int(D)$.  Observe that, among the polynomials $x$ and $f(x)$, $x$ is
indispensable for all $P \in \calP$ and $f(x)$ is indispensable for
all $P\in \calP$ and all $Q\in \calPone$. It follows that, among the
polynomials $x$ and $F(x)$, $x$ is indispensable for all $P \in \calP$
and $F(x)$ is indispensable for all $P\in \calP$ and all
$Q\in \calPone$, cf.~Remark~\ref{remark:indispensable-replacements}.
Hence $G(x)$ is irreducible in $\Int(D)$ by
Lemma~\ref{lemma:factorizations}.

Finally, again by Lemma~\ref{lemma:factorizations} and
Remark~\ref{remark:indispensable-replacements}, $H(x)$ is irreducible
in $\Int(D)$, since
\begin{enumerate}
\item $F(x)$ and $x-a_i$ are indispensable for $P_i$ ($1\le i\le n$)
\item $F(x)$ is indispensable for $Q_i$ ($1\le i\le m$)
\end{enumerate} 
among the polynomials $F(x)$ and $x-a_j$ with $1\le j\le n$.
\end{proof}

As discussed at the beginning of this section, we may conclude:

\begin{corollary}\label{cor:not-transfer-Krull}
  Let $D$ be a Dedekind domain with infinitely many maximal ideals,
  all of them of finite index.
 
  Then there does not exist a transfer homomorphism from the multiplicative
  monoid\/ $\Int(D)\setminus\{0\}$ to a block monoid; in other words:
  $\Int(D)$ is not a transfer Krull domain.
\end{corollary}

\bibliographystyle{plain}
\bibliography{bibliography}

\end{document}